\documentclass[12pt]{amsart}
\usepackage{amsmath,times,epsfig,amssymb,amsbsy,amscd,amsfonts,amstext,color,bm}
\usepackage{enumerate}
\usepackage{placeins}
\usepackage{array}
\usepackage{tabulary}
\usepackage{multirow}
\usepackage{graphicx}
\usepackage{hyperref}
\usepackage{hhline}
\usepackage{tabu}
\usepackage[table]{xcolor}
\usepackage{subcaption}
\usepackage[color,matrix,arrow]{xypic}
\hypersetup{
  colorlinks   = true, 
  urlcolor     = blue, 
  linkcolor    = blue, 
  citecolor   = red 
}

\theoremstyle{plain}
\newtheorem{theorem}{Theorem}[section]
\newtheorem{corollary}[theorem]{Corollary}
\newtheorem{lemma}[theorem]{Lemma}

\makeatletter
    
    \@addtoreset{equation}{section}
  \makeatother

\theoremstyle{definition}

\newtheorem{example}[theorem]{Example}

\theoremstyle{remark}
\newtheorem{remark}[theorem]{Remark}

\newcommand{\A}{\mathcal{A}}

\newcommand{\F}{\mathbb{F}}

\newcommand{\R}{\mathbb{R}}

\newcommand{\scS}{\mathcal{S}}

\newcommand{\Z}{\mathbb{Z}}

\newcommand{\scH}{{\mathcal{H}}}

\newcommand{\M}{\mathcal{M}}

\newcommand{\lcm}{\operatorname{lcm}}

\newcommand{\quasi}{\operatorname{quasi}}

\newcommand{\DP}{{\mathcal{DP}}}
\newcommand{\SG}{{\mathcal{SG}}}

\newcolumntype{K}[1]{>{\centering\arraybackslash}p{#1}}

\allowdisplaybreaks
\begin{document}

\title[Characteristic quasi-polynomials of ideals]{Characteristic quasi-polynomials of ideals and signed graphs of classical root systems}

\date{\today}

\begin{abstract}
With a main tool is signed graphs, we give a full description of the characteristic quasi-polynomials of ideals of classical root systems ($ABCD$) with respect to the integer and root lattices. 
As a result, we obtain a full description of the characteristic polynomials of the toric arrangements defined by these ideals.
As an application, we provide a combinatorial verification to the fact that the characteristic polynomial of every ideal subarrangement factors over  the dual partition of the ideal in the classical cases.
 \end{abstract}

\author{Tan Nhat Tran}
\address{Tan Nhat Tran, Department of Mathematics, Hokkaido University, Kita 10, Nishi 8, Kita-Ku, Sapporo 060-0810, Japan.}
\email{trannhattan@math.sci.hokudai.ac.jp}




\date{\today}
\maketitle



\section{Introduction}
In recent years, the ``finite field method" for studying hyperplane arrangements have been developed, extended and put into practice. 
Roughly speaking, suppose that the real hyperplane arrangement $\A(\R)$ associated to a list $\A$ of elements in $\Z^\ell$ is given, we can take coefficients modulo a positive integer $q$ and get an arrangement $\A(\Z/q\Z)$ of subgroups in $(\Z/q\Z)^\ell$. 
The central theorem in the theory asserts that when $q$ is a sufficiently large prime, the arrangement $\A(\Z/q\Z)$ now is defined over the finite field $\F_q$, and the cardinality of its complement $\#\M(\A; \Z^\ell, \Z/q\Z)$ coincides with $\chi_{\A(\R)}(q)$, the evaluation of the characteristic polynomial $\chi_{\A(\R)}(t)$ of $\A(\R)$ at $q$ (e.g., \cite[Theorem 2.2]{A96}). 
Later on, Kamiya-Takemura-Terao showed that $\#\M(\A; \Z^\ell, \Z/q\Z)$ is actually a quasi-polynomial in $q$  \cite{KTT08}, and left the task of understanding the constituents of this quasi-polynomial to be an interesting problem. 
A number of attempts have been made in order to tackle the problem (e.g., \cite{KTT08}, \cite{LTY17}, \cite{DFM17}), especially, two interpretations for every constituent via subspace and toric viewpoints have been found  \cite{TY18}. 
The mentioning establishments open a new direction for studying the combinatorics and topology of hyperplane and toric arrangements in one single quasi-polynomial.

Much of the motivation for the study of the hyperplane and toric arrangements comes from the arrangements that are defined by irreducible root systems. 
Apart from the theoretical aspects, the ``finite field method" and its toric analogue proved to have efficient applications to compute the characteristic (quasi-)polynomials of several arrangements arising from these vector configurations (e.g., \cite{A96}, \cite{BS98},  \cite{KTT07}, \cite{ACH15}, \cite{Y18L}).  
More concrete computational results have also been derived to assist the observation of interesting coincidences, which we choose to mention some important examples in our study. 
The surprising connection between independent calculations on the Ehrhart quasi-polynomials \cite{Su98} and the characteristic quasi-polynomials \cite{KTT07} produced a main flavor to the analysis on the deformations of root system arrangements in \cite{Y18W}.
Combining the computation on the arithmetic Tutte polynomials of classical root systems \cite{ACH15} with the previously mentioned calculations provided the authors in \cite{LTY17} with the key observation of the identification between the last constituent of the characteristic quasi-polynomial and the corresponding toric arrangement.

Passing from global to local, one may wish to compute the characteristic quasi-polynomials of subsets of a given root system. 
A particularly well-behaved class of the subsets is that of ideals with the associated ideal subarrangements are proved to be free in the sense of Terao \cite{ABCHT16}. 
As a consequence, the characteristic polynomial of every ideal subarrangement factors over the integers with the roots are described combinatorially by the dual partition of the ideal. 
However, some combinatorial explanations for the factorization may have been hidden because of the freeness.
The main goal of this paper is to compute the characteristic quasi-polynomials of the ideals of classical root systems with respect to two different choices of lattices. 
We were inspired and motivated by the ideas and techniques used in \cite{KTT07} that will greatly help us in doing so.
In addition, we wish to provide more combinatorial insights to the understanding of the constituents in connection with the signed graphs.

The remainder of the paper is organized as follows. 
In Section \ref{sec:background}, we recall definitions and basic facts of the characteristic quasi-polynomials, irreducible root systems and their ideals. 
We also recall the constructions of the classical root systems together with the properties of the associated signed graphs.
In Section \ref{sec:ideal}, with a combinatorial ingredient is signed graphs, we compute the characteristic quasi-polynomial of every ideal of a given classical root system with respect to the integer and root lattices. 
As a result, we obtain a full description of the characteristic polynomials of the toric arrangements defined by the ideals.
We will also provide a direct verification to the factorization of the characteristic polynomial of every ideal subarrangement in the classical cases without using the freeness (Theorem \ref{thm:verify}).
\medskip


\section{Preliminaries} 
\label{sec:background}
\subsection{Characteristic quasi-polynomials} 
Let $\Gamma:=\Z^\ell$.
Let $\A$ be a finite list (multiset) of elements in $\Gamma$. 
Let $q\in\Z_{>0}$.
For each $\alpha=(a_1,\ldots,a_\ell) \in \A$, define the subgroup $H_{\alpha, \Z/q\Z}$ of $(\Z/q\Z)^\ell$ by 
$$H_{\alpha, \Z/q\Z}:=\left\{\textbf{z} = (\overline{z_1}, \ldots, \overline{z_\ell})\in  (\Z/q\Z)^\ell  \middle| \sum_{i=1}^\ell a_i \overline{z_i}\equiv\overline{0}\right\}.$$ 
Then the list $\A$ determines the \emph{$q$-reduction} arrangement in $(\Z/q\Z)^\ell$
$$\A(\Bbb \Z/q\Z):=\{H_{\alpha, \Bbb \Z/q\Z} \mid \alpha  \in \A\}.$$
The \emph{complement} of  $\A(\Bbb \Z/q\Z)$ is defined by
$$\M(\A; \Z^\ell,\Z/q\Z):= (\Z/q\Z)^\ell \smallsetminus\bigcup_{\alpha\in\A}H_{\alpha,\Z/q\Z}. $$
For each $\scS\subseteq \A$, write $\Gamma/\langle\scS\rangle\simeq\bigoplus_{i=1}^{n_{\scS}}\Z/d_{\scS, i}\Z\oplus
\Z^{r_{\Gamma}-r_{\scS}}$ where 
$n_{\scS}\geq 0$ and $1<d_{\scS, i}|d_{\scS, i+1}$. 
The \emph{LCM-period} $\rho_{\A}$ of $\A$ is defined by 
\begin{equation*}
\label{eq:LCM-period}
\rho_\A:=\lcm(d_{\scS, n_{\scS}}\mid\scS\subseteq \A). 
\end{equation*} 
It is proved in \cite[Theorem 2.4]{KTT08} that 
$\#\M(\A; \Z^\ell, \Z/q\Z)$ is a monic quasi-polynomial in $q$ for which $\rho_{\A}$ is a period. 
The quasi-polynomial is called the \emph{characteristic quasi-polynomial} of $\A$ (or of $\A(\Z/q\Z)$), and denoted by $\chi^{\quasi}_{\A}(q)$. 
More precisely, there exist monic polynomials 
$f_{\A}^k(t)\in\Z[t]$ ($1 \le k \le \rho_\A$) such that 
for any $q\in\Z_{>0}$ with $q\equiv k\bmod \rho_\A$, 
\begin{equation*}
\chi^{\quasi}_{\A}(q) =f^k_\A(q).
\end{equation*} 
The polynomial $f^k_\A(t)$ is called the \emph{$k$-constituent} of $\chi^{\quasi}_{\A}(q)$. 
It is known that (e.g.,  \cite{A96}, \cite{KTT08}) the $1$-constituent $f^1_\A(t)$ coincides with $\chi_{\A(\R)}(t)$ the characteristic polynomial (e.g., \cite[Definition 2.52]{OT92}) of the real hyperplane arrangement (or $\R$-plexification in the sense of \cite{LTY17}) $\A(\R)=\{H_{\alpha, \Bbb \R} \mid \alpha  \in \A\}$ with 
$H_{\alpha, \R}:=\left\{\textbf{x} \in  \R^\ell  \middle| \sum_{i=1}^\ell a_ix_i=0\right\}$.


\subsection{Root systems and signed graphs}
\label{subsec:Root-systems} 
Our standard reference for root systems is \cite{B68}.
Let $V$ be an $\ell$-dimensional Euclidean space with the standard inner product $(\cdot,\cdot)$. Let $\Phi$ be an irreducible (crystallographic) root system in $V$. 
Fix a positive system $\Phi^+ \subseteq \Phi$ and the associated set of simple roots (base) $\Delta := \{\alpha_1,\ldots,\alpha_\ell \} \subseteq\Phi^+$. 
For $\beta = \sum_{i=1}^\ell n_i\alpha_i \in \Phi^+$, the \emph{height}  of $\beta$ is defined by $ {\rm ht}(\beta) := \sum_{i=1}^\ell n_i $. 

\textit{Notation}: For simplicity of notation, we use the same symbol $M$ for the realization of the matrix $M$ of size $\ell \times m$ as the finite list of elements in $\Gamma=\Z^\ell$ whose elements are the columns of $M$.

For each $\Psi\subseteq\Phi^+$, we assume that an $\ell \times \#\Psi$ integral matrix $S_\Psi = [S_{ij}]$ satisfies
$$\Psi = \left\{\sum_{i=1}^\ell S_{ij}\alpha_i \,\middle|\, 1 \le j \le \#\Psi \right\}.$$
In other words, $S_\Psi$ is the coefficient matrix of $\Psi$ with respect to the base $\Delta$. 
Denote $(\Z/q\Z)^\times:=\Z/q\Z \smallsetminus \{\overline{0}\}$.
We then call $\chi^{\quasi}_{S_\Psi}(\Phi, q)$ the characteristic quasi-polynomial of $\Psi$ with respect to the \emph{root lattice}, and interpret it by the formula 
\begin{equation*}
\chi^{\quasi}_{S_\Psi}(\Phi, q)=\#\{\textbf{z}\in  (\Z/q\Z)^\ell  \mid\textbf{z}\cdot S_\Psi \in ((\Z/q\Z)^\times)^{\#\Psi}\}. 
\end{equation*}

We define $\scH_{\Psi}:= \{H_\alpha \mid \alpha\in\Psi\}$, where $H_\alpha=\{x\in V \mid (\alpha,x)=0\}$ is the hyperplane orthogonal to $\alpha$. 
It is not hard to see that $\scH_{\Psi}$ is the $\R$-plexification of $S_\Psi$ i.e., $\scH_{\Psi}=S_\Psi(\R)$. 
Note also that $\scH_{\Phi^+}$ is called the \emph{Weyl arrangement} of $\Phi^+$, and $\scH_{\Psi}$ is a Weyl subarrangement.

In the remainder of the paper, we are mainly interested in the root system $\Phi$ of classical type ($ABCD$).
Let us recall briefly the constructions of these root systems\footnote{We decided to omit the construction of type $A$ root systems as the calculation on this type follows from those on the other types (e.g., see formula \eqref{eq:type-A}).} following \cite[Chapter VI, $\S$4]{B68}.
Let $\{\epsilon_1, \ldots, \epsilon_{\ell}\}$ be an orthonormal basis for $V$. 
If  $\ell \ge 2$ then
$$\Phi(B_\ell)  = \{\pm\epsilon_i \,(1 \le i  \le \ell),\pm(\epsilon_i \pm \epsilon_j)  \,(1 \le i < j \le \ell )\},$$
with $\#\Phi(B_\ell) =2\ell^2$ is an irreducible root system in $V$ of type $B_{\ell}$. 
We may choose a positive system
$$\Phi^+(B_\ell)  = \{\epsilon_i \,(1 \le i  \le \ell), \epsilon_i \pm \epsilon_j \,(1 \le i < j \le \ell )\}.$$
Define $\alpha_i := \epsilon_i - \epsilon_{i+1}$, for $1 \le i  \le \ell-1$, and $\alpha_{\ell} := \epsilon_{\ell}$. 
Then $\Delta(B_\ell)  = \{\alpha_1,\ldots, \alpha_{\ell}\}$ is the base associated with $\Phi^+(B_\ell) $. 
We may express
\begin{align*}
\Phi^+(B_\ell)  & =  \{ \epsilon_i=\sum_{i\le k \le \ell }\alpha_k \,(1 \le i \le \ell), \epsilon_i-\epsilon_j=\sum_{i\le k < j }\alpha_k \,(1 \le i <  j \le \ell), \\
& \epsilon_i+\epsilon_j=\sum_{i\le k<j }\alpha_k+2 \sum_{j\le k \le \ell }\alpha_k\,(1 \le i < j \le \ell) \}.
\end{align*}
For any $\Psi\subseteq\Phi^+(B_\ell)$, we write $T_\Psi = [T_{ij}]$ for the coefficient matrix of $\Psi$ with respect to the orthonormal basis. 
We then call $\chi^{\quasi}_{T_\Psi}(\Phi, q)$ the characteristic quasi-polynomial of $\Psi$ with respect to the \emph{integer lattice}.
The matrices $T_\Psi$ and $S_\Psi$ are related by $T_\Psi=P(B_\ell)\cdot S_\Psi$, where $P(B_\ell) $ is an unimodular matrix of size $\ell \times \ell$ given by
$$P(B_\ell) =\begin{bmatrix}
1 &   &   & &  \\
-1 & 1 &   &  & \\
  & -1 &   &  & \\
  &   & \ddots  &  & \\
    &   &  & 1 & \\
      &   &  & -1 & 1
\end{bmatrix}.$$
Similarly, let $\ell \ge 2$, an irreducible root system of type $C_{\ell}$ is given by
\begin{align*}
\Phi(C_\ell)  &= \{\pm2\epsilon_i \,(1 \le i  \le \ell),\pm(\epsilon_i \pm \epsilon_j)  \,(1 \le i < j \le \ell )\}, \\
\Phi^+(C_\ell)  &= \{2\epsilon_i \,(1 \le i  \le \ell), \epsilon_i \pm \epsilon_j \,(1 \le i < j \le \ell )\}, \\
\Delta(C_\ell) &=\{\alpha_i = \epsilon_i - \epsilon_{i+1}\,(1 \le i  \le \ell-1),\, \alpha_{\ell} =2\epsilon_{\ell} \} ,\\
\Phi^+(C_\ell)  &= \{ 2\epsilon_i=2\sum_{i\le k< \ell }\alpha_k+\alpha_\ell \,(1 \le i \le \ell), \epsilon_i-\epsilon_j=\sum_{i\le k<j }\alpha_k \,(1 \le i < j \le \ell),  \\
& \epsilon_i+\epsilon_j=\sum_{i\le k<j }\alpha_k+2 \sum_{j\le k <\ell }\alpha_k+\alpha_\ell\,(1 \le i < j \le \ell) \}.
\end{align*}
Finally, let $\ell \ge 3$, an irreducible root system of type $D_{\ell}$ is given by
\begin{align*}
\Phi(D_\ell)  &= \{\pm(\epsilon_i \pm \epsilon_j)  \,(1 \le i < j \le \ell )\}, \\
\Phi^+(D_\ell)  &= \{ \epsilon_i \pm \epsilon_j \,(1 \le i < j \le \ell )\}, \\
\Delta(D_\ell) &=\{\alpha_i = \epsilon_i - \epsilon_{i+1}\,(1 \le i  \le \ell-1),\, \alpha_{\ell} =\epsilon_{\ell-1}+\epsilon_{\ell} \},\\
\Phi^+(D_\ell)  &=\{ \epsilon_i + \epsilon_{\ell}=\sum_{i\le k \le \ell-2 }\alpha_k+\alpha_\ell \,(1 \le i< \ell), \\
& \epsilon_i - \epsilon_{j}=\sum_{i< k < j }\alpha_k \,(1 \le i <  j \le \ell), \\
& \epsilon_i + \epsilon_{j}=\sum_{i\le k<j }\alpha_k+2 \sum_{j\le k <\ell-1 }\alpha_k+\alpha_{\ell-1}+\alpha_\ell\,(1 \le i < j< \ell) \}.
\end{align*}
From the constructions above, we obtain the comparison of the height placements of positive roots in $\Phi(B_\ell)$, $\Phi(C_\ell)$ and $\Phi(D_\ell)$ as in Table \ref{tab:placement}.

 \begin{table}[htbp]
\centering
{\footnotesize\renewcommand\arraystretch{1} 
\begin{tabular}{K{2.5cm}|K{2.5cm}|K{2.5cm}|K{2.5cm}}
Root & Height in $B_\ell$ & Height in $C_\ell$ & Height in $D_\ell$ \\
\hline\hline
$\begin{array}{c}
     \epsilon_i    \\
    (1 \le i \le \ell)
    \end{array}$
& $\ell-i+1$ &None &None \\
\hline
$\begin{array}{c}
     2\epsilon_i    \\
    (1 \le i \le \ell)
    \end{array}$
& None &$2(\ell-i)+1$ &None \\
\hline
$\begin{array}{c}
     \epsilon_i + \epsilon_j\\
    (1 \le i < j \le \ell)
    \end{array}$
& $2\ell-i-j+2$ & $2\ell-i-j+1$ &$2\ell-i-j$ \\
\hline
$\begin{array}{c}
     \epsilon_i - \epsilon_j  \\
    (1 \le i<j \le \ell)
    \end{array}$
& $j-i$ & $j-i$ &$j-i$
\end{tabular}
}
\bigskip
\caption{Height placements in $\Phi(B_\ell)$, $\Phi(C_\ell)$ and $\Phi(D_\ell)$.}
\label{tab:placement}
\end{table}

In the language of \emph{signed graphs} following \cite[\S5]{Z81}, we can associate to each subset $\Psi\subseteq\Phi^+(B_\ell)$ a signed graph $G:=G(\Psi)=(V_G, E_{G^+},E_{G^-},L_G)$ on the vertex set 
$$V_G:=\{v_i,v_j \mid \mbox{$\epsilon_i  \in \Psi$ or $\epsilon_i-\epsilon_{j} \in \Psi$ or $\epsilon_i+\epsilon_{j} \in \Psi$} \},$$ with the set of positive edges $E_{G^+}:=\{e_{ij}^+ \mid \epsilon_i + \epsilon_{j} \in \Psi\}$, the set of negative edges $E_{G^-}:=\{e_{ij}^- \mid \epsilon_i - \epsilon_{j} \in \Psi\}$, and the set of loops $L_G:=\{\ell_i \mid \epsilon_i \in \Psi\}$. 
Alternatively, if $\Psi\subseteq\Phi^+(C_\ell)$, we can define $L_G:=\{\ell_i \mid 2\epsilon_i \in \Psi\}$. 
To extract information from $\Psi$ by using $G(\Psi)$, we associate to it an unordered sequence of nonnegative integers, denoted $\SG(\Psi):=(p_1,\ldots,p_\ell)$ with for each $i$ ($1 \le i \le \ell$)
\begin{equation*}\label{eq:di-signed-graphs}
p_i :=\# \{e_{ij}^+ \mid e_{ij}^+ \in E_{G^+}\}+ \# \{e_{ij}^- \mid e_{ij}^- \in E_{G^-}\} + \# \{\ell_{i} \mid \ell_{i} \in L_{G}\}.
\end{equation*}

\subsection{Ideals}
\label{subsec:ideals} 
Define the partial order $\succeq$ on $\Phi^+$ such that $\beta_1 \succeq \beta_2$ if and only if $\beta_1-\beta_2 = \sum_{i=1}^\ell n_i\alpha_i$ with all $n_i \in \Z_{\ge 0}$. 
A subset $I$ of  $ \Phi^+$ is called an \emph{ideal} if, for $\beta_1,\beta_2 \in \Phi^+$, $\beta_1 \succeq \beta_2, \beta_ 1 \in I$ then $\beta_2 \in I$. 
Let us recall the recent advance towards the study of the ideals. 
Let $\Theta^{(k)} \subseteq \Phi^+$ be the set consisting of positive roots of height $k$. 
Let $I$ be an ideal of  $ \Phi^+$ and set $M:=\max\{{\rm ht}(\beta)\mid \beta \in I\}$. 
The \textit{height distribution} of $I$ is defined as a sequence of positive integers:
$$(i_1, \ldots , i_k, \ldots , i_{M}),$$ where 
$i_k := \#\Theta^{(k)}$ for $1 \le k \le M$.
The  \textit{dual partition} $\DP(I)$ of (the height distribution of) $I$ is given by a sequence of nonnegative integers:
$$\DP(I) := \left( (0)^{\ell-i_1},(1)^{i_1-i_2},\ldots ,(M-1)^{i_{M-1}-i_{M}},(M)^{i_{M}}\right),$$ 
where notation $(a)^b$ means the integer $a$ appears exactly $b$ times.  
Although the definition of the dual partition seems to esteem the (increasing) order of components in the sequence, this requirement is not important in this paper. 
Two dual partitions of an ideal are conventionally identical if the partitions differ only by a re-ordering of the components.
\begin{theorem}[\cite{ABCHT16}]\label{thm:dual}
 Any ideal subarrangement $\scH_I$ is free (in the sense of Terao) with the set of exponents coincides with $\DP(I)$.
\end{theorem}

\begin{corollary}[\cite{ABCHT16}]\label{cor:ideal-factorization}   
 For any ideal $I\subseteq \Phi^+$, the characteristic polynomial $\chi_{\scH_I}(\Phi, t)$ factors as follows:
\begin{equation*}\label{eq:deal-factorization}
\chi_{\scH_I}(\Phi,t)= \prod_{i=1}^\ell (t-d_i),
\end{equation*}
where $\DP(I)=(d_1,\ldots,d_\ell)$.
\end{corollary}
 
When $\Phi$ is of type $B_\ell$ or $C_\ell$, $\DP(I)=\SG(I)$, while $\DP(I)\ne\SG(I)$ if $\Phi$ is of type $D_\ell$. 
Also, $\SG(I)$ does not determine $I$, for instance, $I_1=\{\epsilon_4 - \epsilon_5\}$ and $I_2=\{\epsilon_4 + \epsilon_5\}$ are distinct ideals of  $\Phi^+(D_5)$, but $\SG(I_1)=\SG(I_2)=(0,0,0,1,0)$. 

\section{Computation on ideals} 
\label{sec:ideal} 
In the remainder of the paper, we assume that $\Phi$ is of classical type.
We summarize some easy cases that the computation of the characteristic quasi-polynomials is manageable thanks to Corollary \ref{cor:ideal-factorization}.
The minimum period coincides with the LCM-period \cite[Remark 3.3]{KTT10}. 
So the minimum period of $\chi^{\quasi}_{S_I}(A_\ell, q)$ is $1$ for every $I\subseteq\Phi^+$; hence $\chi^{\quasi}_{S_I}(A_\ell, q)=\chi_{\scH_I}(A_\ell,q)$.
For other cases, the minimum period of $\chi^{\quasi}_{S_I}(\Phi,q)$ is at most $2$; hence  we know the $1$-constituents: $f^1_{S_I}(\Phi,t)=\chi_{\scH_I}(\Phi,t)$. 
We are left with the task of determining  $f^2_{S_I}(\Phi, t)$, or equivalently, $\chi^{\quasi}_{S_I}(\Phi, q)$ when $q$ is even, and $\Phi$ is of type $B$, $C$ or $D$. 
Turning the problem around, we would like to verify Corollary \ref{cor:ideal-factorization} by using the information of ideals via signed graphs without relying on the freeness, which we will do in Theorem \ref{thm:verify}. 
 
\subsection{Type $B$ root systems}
By \cite[Theorem 4.1]{KTT07}\footnote{Note that the number $b(q)$ in \cite[Theorem 4.1]{KTT07} should be read as $b(q)=\prod_{i=1}^m\gcd\{q,d_i\}$, where $d_i$'s are invariant factors of the matrix $P$. 
In particular, if $\det(P)$ takes the value in $\{1,2\}$, then $b(q)=\gcd\{q,\det(P)\}$. 
Hence \cite[Theorem 4.1]{KTT07} is valid if $\Phi$ is a classical root system.}, if $\Psi\subseteq\Phi^+(B_\ell)$, 
\begin{equation*}\label{eq:B-TS}
\chi^{\quasi}_{S_\Psi}(B_\ell, q)=\chi^{\quasi}_{T_\Psi}(B_\ell, q).
\end{equation*}

Let $I$ be an ideal of  $\Phi^+(B_\ell)$. 
Assume that $\DP(I)=\SG(I)=(d_1,\ldots,d_\ell)$. For each $k$ $(1 \le k \le \ell)$, write $d_k=d_k^{(+)}+d_k^{(-)}+d_k^{(0)}$ for a partition of $d_k$ with
\begin{equation*}
\label{eq:partition-dk}
\begin{aligned}
d_k^{(0)} & := 
\begin{cases}
0 \quad\mbox{ if $\epsilon_k  \notin I$}, \\ 
1 \quad\mbox{ if $\epsilon_k   \in I$}.
\end{cases}
\\
d_k^{(\pm)} & :=\#\{\epsilon_k \pm\epsilon_{j} \mid \epsilon_k \pm\epsilon_{j} \in I \}.
\end{aligned}
\end{equation*}
The partitions give a partition of $I$ which we call it the \emph{$B$-partition}, as follows: $I =I^{0} \sqcup I^{-} \sqcup I^{+},$ where
\begin{equation*}
\label{eq:B-partition}
\begin{aligned}
I^{0} & := \{\epsilon_i \mid \epsilon_i   \in I\}\\
I^{\pm} & := \{\epsilon_i\pm\epsilon_j \mid \epsilon_i \pm\epsilon_{j} \in I \}.
\end{aligned}
\end{equation*}

If $\epsilon_i + \epsilon_{j} \notin I$ for all $i, j$ (type $A$), then for all $q \in \Z_{>0}$, 
\begin{equation}
\label{eq:type-A}
\chi^{\quasi}_{T_I}(B_\ell, q) = \prod_{i=1}^\ell (q-d_i).
\end{equation}
Now assume that some $\epsilon_i + \epsilon_{j} \in I$ with $1 \le i < j \le \ell$. 
In particular, $\epsilon_k \in I$ for all $i \le k \le \ell$. 
Set $s:=\min \{ 1 \le k \le \ell \mid \epsilon_k \in I\}$. 
Denote $R:=I \smallsetminus \{\epsilon_i - \epsilon_{j} \in I \mid 1\le i<s, i <j \le \ell\}$. 
Thus $R$ is an ideal of the root subsystem of $\Phi(B_\ell)$ of  type $B_{\ell-s+1}$ with a base given by $\Delta(B_{\ell-s+1})=\{\alpha_s,\ldots,\alpha_\ell\}$. 
Furthermore, for all $q \in \Z_{>0}$, we have
\begin{equation*}\label{eq:inductive-B}
\chi^{\quasi}_{T_I}(B_\ell, q) =\chi^{\quasi}_{T_R}(B_{\ell-s+1}, q)\cdot \prod_{i=1}^{s-1}(q-d_i).
\end{equation*}
Then it suffices to consider $s=1$ i.e., $\epsilon_1 \in I$. 
For such ideals, $d_k^{(-)} = \ell-k$, $d_k^{(0)}=1$ for $1 \le k \le \ell$.

 \begin{lemma}
 \label{lem:B-D}
 Let $I$ be an ideal of $\Phi^+(B_\ell)$ with $\epsilon_1 \in I$. 
 Set $J:=I \smallsetminus I^{0}$. 
 \begin{enumerate}[(a)]
\item  $J$ is an ideal of $\Phi^+(D_\ell)$.
\item  $\DP(J)=(p_1,\ldots,p_\ell)$ with $p_k=d_k^{(-)}+d_{k-1}^{(+)}$ for all $1 \le k \le \ell$. 
Here we agree that $d_{0}^{(2)} \equiv 0$.
\end{enumerate}
\end{lemma} 
\begin{proof}
The proof of (a) is straightforward by the definition of ideals. 
The proof of (b) follows from the height placements in Table \ref{tab:placement}.
 \end{proof}

\begin{theorem} \label{thm:B-and-D} 
Under the Lemma \ref{lem:B-D}'s assumptions, if $q \in \Z_{>0}$ is even, 
  \begin{equation*}\label{eq:B-D}
\chi^{\quasi}_{T_{I}}(B_\ell, q)=\chi^{\quasi}_{T_J}(D_\ell, q-1)= \prod_{i=1}^\ell (q-p_i-1).
\end{equation*}
\end{theorem}
\begin{proof} 
The proof of the first equality is similar to (but more general than) that of \cite[Lemma 4.4(11)]{KTT07}. 
\begin{align*}
\chi^{\quasi}_{T_I}(B_\ell, q) 
&=\# \left\{ \textbf{z}  \in   (\Z/q\Z)^\ell   \middle|
\begin{array}{c}
      \overline{z_i} \ne  \overline{z_j}\,(1 \le i <j \le \ell),  \\
      \overline{z_i} +  \overline{z_j}\ne \overline0\,(\epsilon_i+\epsilon_{j} \in I), \\
      \overline{z_i}\ne \overline0\,(1 \le i \le \ell)
    \end{array}
\right\} \\
&= \#\left\{(z_1,\ldots,z_\ell)  \in \Z^\ell  \middle|
\begin{array}{c}
     z_i \ne z_j \,(1 \le i <j \le \ell),  \\
      z_i +z_j \ne q\,(\epsilon_i+\epsilon_{j} \in I), \\
     1 \le z_i \le q-1\,(1 \le i \le \ell)
    \end{array}
\right\} \\
&= \#\left\{(v_1,\ldots,v_\ell)  \in \Z^\ell  \middle|
\begin{array}{c}
     v_i \ne v_j \,(1 \le i <j \le \ell),  \\ 
     v_i +v_j \ne 0\,(\epsilon_i+\epsilon_{j} \in I), \\
     -\frac{q-2}2 \le v_i \le \frac{q-2}2\,(1 \le i \le \ell)
    \end{array}
\right\} \\
&= \#\left\{(t_1,\ldots,t_\ell)  \in \Z^\ell  \middle|
\begin{array}{c}
     t_i \ne t_j\,(1 \le i <j \le \ell),  \\ 
t_i +t_j \ne q-1\,(\epsilon_i+\epsilon_{j} \in I), \\
     0 \le t_i \le q-2\,(1 \le i \le \ell)
    \end{array}
\right\} \\
&=\# \left\{ \textbf{u}  \in   (\Z/(q-1)\Z)^\ell   \middle|
\begin{array}{c}
      \overline{u_i} \ne  \overline{u_j}\,(1 \le i <j \le \ell),  \\
      \overline{u_i} +  \overline{u_j}\ne \overline0\,(\epsilon_i+\epsilon_{j} \in J)
          \end{array}
\right\}\\
&= \chi^{\quasi}_{T_J}(D_\ell, q-1).
\end{align*} 
We have used the following changes of variables
$$v_i =z_i- \frac{q}2 \quad \mbox{and}\quad t_i =
\begin{cases}
v_i \mbox{ if } v_i \ge 0, \\
v_i+q-1 \mbox{ if } v_i< 0.
\end{cases}
$$
The second equality follows from Lemma \ref{lem:B-D} and Corollary \ref{cor:ideal-factorization}.
 \end{proof} 
 
\begin{example}\label{eg:B-parition}   
 Table \ref{tab:B5-Bpartition} shows the $B$-partition of an ideal $I=\{\alpha \in \Phi^+(B_5) \mid {\rm ht}(\alpha) \le 7\}$ (in colored region), with  $I^{0}, I^{-}, I^{+}$ are colored in red, yellow, blue, respectively. 
 Table \ref{tab:D5-partition-J} shows the corresponding partition of the ideal $J=I \smallsetminus I^{0}$ in $\Phi^+(D_5)$. 
 In this case, $\DP(I)=(7,7,5,3,1)$ and $\DP(J)=(4,5,5,3,1)$. Hence for even $q \in \Z_{>0}$, we have
   \begin{equation*} 
\chi^{\quasi}_{T_{I}}(B_\ell, q)=(q-2)(q-4)(q-5)(q-6)^2.
\end{equation*}
\end{example}

\begin{table}[htbp]
\centering
{\footnotesize\renewcommand\arraystretch{1.5} 
\begin{tabular}{ccccccc}
\mbox{Height} & & & && \\
9 &   $\epsilon_1+\epsilon_2$ &   &  & && \\
8 &  $\epsilon_1+\epsilon_3$ &   & & &&\\
\hhline{~--}
7 & \multicolumn{1}{|c}{\cellcolor{blue!25}{$\epsilon_1+\epsilon_4$}} & \multicolumn{1}{c|}{\cellcolor{blue!25}{$\epsilon_2+\epsilon_3$}}  & && \\
6 & \multicolumn{1}{|c}{\cellcolor{blue!25}{$\epsilon_1+\epsilon_5$}} & \multicolumn{1}{c|}{\cellcolor{blue!25}{$\epsilon_2+\epsilon_4$}}  & && \\
\cline{4-4}
5 & \multicolumn{1}{|c}{\cellcolor{red!25}{$\epsilon_1$}} & \cellcolor{blue!25}{$\epsilon_2+\epsilon_5$} &\multicolumn{1}{c|}{\cellcolor{blue!25}{$\epsilon_3 +\epsilon_4$}} &&& \\
4& \multicolumn{1}{|c}{\cellcolor{yellow!25}$\epsilon_1-\epsilon_5$} & \cellcolor{red!25}{$\epsilon_2$} & \multicolumn{1}{c|}{\cellcolor{blue!25}{$\epsilon_3 +\epsilon_5$}} &&&\\
\cline{5-5}
3 & \multicolumn{1}{|c}{\cellcolor{yellow!25}$\epsilon_1-\epsilon_4$} & \cellcolor{yellow!25}$\epsilon_2-\epsilon_5$ & \cellcolor{red!25}{$\epsilon_3$} &  \multicolumn{1}{c|}{\cellcolor{blue!25}{$\epsilon_4+\epsilon_5$}} &&\\
2& \multicolumn{1}{|c}{\cellcolor{yellow!25}$\epsilon_1-\epsilon_3$} & \cellcolor{yellow!25}$\epsilon_2-\epsilon_4$ & \cellcolor{yellow!25}$\epsilon_3-\epsilon_5$ & \multicolumn{1}{c|}{\cellcolor{red!25}{$\epsilon_4$}} &&\\
\cline{6-6}
1 & \multicolumn{1}{|c}{\cellcolor{yellow!25}$\epsilon_1-\epsilon_2$} & \cellcolor{yellow!25}$\epsilon_2 -\epsilon_3$& \cellcolor{yellow!25}$\epsilon_3-\epsilon_4$  & \cellcolor{yellow!25}$\epsilon_4-\epsilon_5$ & \multicolumn{1}{c|}{\cellcolor{red!25}{$\epsilon_5$}} &\\
\cline{2-6}
& 7 & 7 & 5 & 3  & 1 & $\DP(I)$
\end{tabular}
}
\bigskip
\caption{The $B$-partition of an ideal $I$ in $\Phi^+(B_5)$.}
\label{tab:B5-Bpartition}
\end{table}

\begin{table}[htbp]
\centering
{\footnotesize\renewcommand\arraystretch{1.5} 
\begin{tabular}{ccccccc}
\mbox{Height} & & & & &\\
7&   & $\epsilon_1+\epsilon_2$   &  & && \\
6&  & $\epsilon_1+\epsilon_3$    & & &&\\
\hhline{~~--}
5 & & \multicolumn{1}{|c}{\cellcolor{blue!25}{$\epsilon_1+\epsilon_4$}} & \multicolumn{1}{c|}{\cellcolor{blue!25}{$\epsilon_2+\epsilon_3$}}   && \\
\cline{2-2}
4& \multicolumn{1}{|c}{\cellcolor{yellow!25}$\epsilon_1-\epsilon_5$} &  \cellcolor{blue!25}{$\epsilon_1+\epsilon_5$}  & \multicolumn{1}{c|}{\cellcolor{blue!25}{$\epsilon_2+\epsilon_4$}} &&&\\
\cline{5-5}
3 & \multicolumn{1}{|c}{\cellcolor{yellow!25}$\epsilon_1-\epsilon_4$} & \cellcolor{yellow!25}$\epsilon_2-\epsilon_5$  &  \cellcolor{blue!25}{$\epsilon_2+\epsilon_5$} &\multicolumn{1}{c|}{\cellcolor{blue!25}{$\epsilon_3 +\epsilon_4$}}  &&\\
2& \multicolumn{1}{|c}{\cellcolor{yellow!25}$\epsilon_1-\epsilon_3$} & \cellcolor{yellow!25}$\epsilon_2-\epsilon_4$ &  \cellcolor{yellow!25}$\epsilon_3-\epsilon_5$  &\multicolumn{1}{c|}{\cellcolor{blue!25}{$\epsilon_3 +\epsilon_5$}}&& \\
\cline{6-6}
1 & \multicolumn{1}{|c}{\cellcolor{yellow!25}$\epsilon_1-\epsilon_2$} & \cellcolor{yellow!25}$\epsilon_2 -\epsilon_3$& \cellcolor{yellow!25}$\epsilon_3-\epsilon_4$  &  \cellcolor{yellow!25}$\epsilon_4-\epsilon_5$ & \multicolumn{1}{c|}{\cellcolor{blue!25}{$\epsilon_4+\epsilon_5$}} &\\
\cline{2-6}
& 4 & 5 & 5 & 3  & 1 & $\DP(J)$
\end{tabular}
}
\bigskip
\caption{The resulting partition of $J=I \smallsetminus I^{0}$ in $\Phi^+(D_5)$.}
\label{tab:D5-partition-J}
\end{table}

\begin{remark}
\label{rem:B-recover}
When $I=\Phi^+(B_\ell)$, $\DP(I)=(2\ell-1,2\ell-3,\ldots,3,1)$ and $\DP(J)=(\ell-1,2\ell-3,\ldots,3,1)$. 
Thus for even $q \in \Z_{>0}$, we have
   \begin{equation*} 
\chi^{\quasi}_{S_{\Phi^+}}(B_\ell, q)=\chi^{\quasi}_{T_{\Phi^+}}(B_\ell, q)=(q-2)(q-4)\ldots(q-(2\ell-2))(q-\ell),
\end{equation*}
which recovers the result of \cite[Theorem 4.8]{KTT07} for type $B$ root systems.
 \end{remark}

\subsection{Type $C$ root systems}
By \cite[Theorem 4.1]{KTT07}, for any $\Psi\subseteq\Phi^+(C_\ell)$ and for even $q \in \Z_{>0}$,  we have
\begin{equation*}\label{eq:C-TS}
\chi^{\quasi}_{S_\Psi}(C_\ell, q)=\frac12\left( \chi^{\quasi}_{T_\Psi}(C_\ell, q)+F_\Psi(C_\ell, q)\right),
\end{equation*}
\begin{align*}
F_\Psi(C_\ell, q) & :=\# \{ \textbf{z}\in  (\Z/q\Z)^\ell  \mid \textbf{z}\cdot T_\Psi+\textbf{g}\cdot S_\Psi\in ((\Z/q\Z)^\times)^{\#\Psi}\}, \\
\textbf{g} & := (\overline0,\overline0,\ldots,\overline1) \in  (\Z/q\Z)^\ell .
\end{align*} 
Let $I$ be an ideal of  $\Phi^+(C_\ell)$ with $\DP(I)=\SG(I)=(d_1,\ldots,d_\ell)$.
We need only consider $\epsilon_i + \epsilon_{j} \in I$ for some $1 \le i < j \le \ell$. 
In particular, $2\epsilon_k \in I$ for all $j \le k \le \ell$. 
Set $s:=\min \{ 1 \le k \le \ell \mid 2\epsilon_k \in I\}$. 
Define $R:=I \smallsetminus \{ \epsilon_i \pm \epsilon_{j} \in I \mid 1\le i<s, i <j \le \ell\}$. 
Thus $R$ itself is the positive system of a root system of type $C_{\ell-s+1}$ with a base given by $\Delta(C_{\ell-s+1})=\{\alpha_s,\ldots,\alpha_\ell\}$. 
Furthermore, for all $q \in \Z_{>0}$, we have
\begin{equation*}\label{eq:inductive-C}
\begin{aligned}
\chi^{\quasi}_{T_I}(C_\ell, q) & =\chi^{\quasi}_{T_R}(C_{\ell-s+1}, q)\cdot \prod_{i=1}^{s-1}(q-d_i),\\
F_I(C_\ell, q) &= F_R(C_{\ell-s+1}, q)\cdot \prod_{i=1}^{s-1}(q-d_i).
\end{aligned}
\end{equation*}
Then it suffices to consider $s=1$ or equivalently, $I=\Phi^+(C_\ell)$. 
The computations of $\chi^{\quasi}_{T_{\Phi^+}}(C_\ell, q)$, $F_{\Phi^+}(C_\ell, q)$ and $\chi^{\quasi}_{S_{\Phi^+}}(C_\ell, q)$ were already done in 
\cite[Theorem 4.7 and \S4.3]{KTT07}. 
More direct computations are also obtainable. 
For instance, when $q$ is even, we have
\begin{align*}
\chi^{\quasi}_{T_{\Phi^+}}(C_\ell, q) & =\# \{ \textbf{z}  \in   (\Z/q\Z)^\ell  \mid \overline{z_i} \notin \{ \overline0, \overline{q/2},\pm\overline{z_j}\}, 1 \le i <j \le \ell \}\\ 
&= \prod_{i=1}^{\ell}(q-(d_i+1)).
\end{align*}

\begin{example}\label{eg:C-ideal}   
 Table \ref{tab:C5-ideal} shows an example of an ideal $I\subsetneq  \Phi^+(C_5)$ (in enclosed region). 
 In this case, $\DP(I)=(4,6,5,3,1)$. Hence for even $q \in \Z_{>0}$, we have
\begin{align*}
\chi^{\quasi}_{T_{I}}(C_\ell, q) & =(q-6)^2(q-4)^2(q-2),\\ 
F_{\Phi^+}(C_\ell, q) &= (q-6)(q-4)^2(q-2)q,\\ 
\chi^{\quasi}_{S_{I}}(C_\ell, q) & = (q-6)(q-4)^2(q-3)(q-2).
\end{align*}
\end{example}

\begin{table}[htbp]
\centering
{\footnotesize\renewcommand\arraystretch{1.5} 
\begin{tabular}{ccccccc}
\mbox{Height} & & & && \\
9 &  $2\epsilon_1$ &   &  & && \\
8 & $\epsilon_1+\epsilon_2$ &   & & &&\\
7 &$\epsilon_1+\epsilon_3$ & $2\epsilon_2$  & && \\
\cline{3-3}
6 & $\epsilon_1+\epsilon_4$ & \multicolumn{1}{|c|}{$\epsilon_2+\epsilon_3$}  & && \\
 \cline{4-4}
5 &$\epsilon_1+\epsilon_5$ & \multicolumn{1}{|c}{$\epsilon_2+\epsilon_4$} &\multicolumn{1}{c|}{$2\epsilon_3$} &&& \\
\cline{2-2}
4& \multicolumn{1}{|c}{$\epsilon_1-\epsilon_5$} & $\epsilon_2+\epsilon_5$& \multicolumn{1}{c|}{$\epsilon_3 +\epsilon_4$} &&&\\
\cline{5-5}
3 & \multicolumn{1}{|c}{$\epsilon_1-\epsilon_4$} & $\epsilon_2-\epsilon_5$ &$\epsilon_3 +\epsilon_5$ &  \multicolumn{1}{c|}{$2\epsilon_4$} &&\\
2& \multicolumn{1}{|c}{$\epsilon_1-\epsilon_3$} & $\epsilon_2-\epsilon_4$ & $\epsilon_3-\epsilon_5$ &\multicolumn{1}{c|}{$\epsilon_4+\epsilon_5$}&&\\
\cline{6-6}
1 & \multicolumn{1}{|c}{$\epsilon_1-\epsilon_2$} & $\epsilon_2 -\epsilon_3$& $\epsilon_3-\epsilon_4$  &  $\epsilon_4-\epsilon_5$ & \multicolumn{1}{c|}{$2\epsilon_5$}  &\\
\cline{2-6}
& 4 & 6 & 5 & 3  & 1 & $\DP(I)$
\end{tabular}
}
\bigskip
\caption{An ideal $I$ in $\Phi^+(C_5)$.}
\label{tab:C5-ideal}
\end{table}

\subsection{Type $D$ root systems} 
The computation on this type requires a bit more effort.
By \cite[Theorem 4.1]{KTT07}, if $\Psi\subseteq\Phi^+(D_\ell)$ and $q$ is even,
\begin{equation}\label{eq:D-TS}
\chi^{\quasi}_{S_\Psi}(D_\ell, q)=\frac12\left( \chi^{\quasi}_{T_\Psi}(D_\ell, q)+F_\Psi(D_\ell, q)\right),
\end{equation}
\begin{align*}
F_\Psi(D_\ell, q) & :=\# \{ \textbf{z}\in  (\Z/q\Z)^\ell  \mid \textbf{z}\cdot T_\Psi+\textbf{g}\cdot S_\Psi\in ((\Z/q\Z)^\times)^{\#\Psi}\}, \\
\textbf{g} & := (\overline0,\overline0,\ldots,\overline1) \in  (\Z/q\Z)^\ell .
\end{align*} 

Let $I$ be an ideal of  $\Phi^+(D_\ell)$. 
We need only consider $\epsilon_i + \epsilon_{j} \in I$ for some $1 \le i < j \le \ell$. 
In particular, $\epsilon_{\ell-1}+\epsilon_\ell \in I$. 
Define
\begin{equation}\label{eq:s-D}
s :=\min \{ 2 \le k \le \ell \mid \epsilon_{k-1}+\epsilon_k  \in I\}.
\end{equation} 
If $\epsilon_{\ell-1}-\epsilon_{\ell}  \notin I$, we must have $s=\ell$. 
Then the computation can be reduced (up to a bijection) to that on the type $A$ root systems, which can be done easily.
Suppose henceforth that $\epsilon_{\ell-1}-\epsilon_{\ell}  \in I$, and set
\begin{equation*}\label{eq:r-D}
r :=\min \{ 1 \le k \le \ell \mid \epsilon_{k}+\epsilon_{\ell}  \in I \,\,\mbox{and}\,\, \epsilon_{k}-\epsilon_{\ell}  \in I\}.
\end{equation*} 
Obviously, $r \le s-1$. 
Assume that $\SG(I)=(p_1,\ldots,p_\ell)$ with for each $i$
\begin{equation}\label{eq:di-signed-graphs-D}
p_i = \# \{ \epsilon_i - \mu_{i,j}\epsilon_{j} \in I \mid \mu_{i,j} \in \{\pm1\}\}.
\end{equation} 
Define $R:=I \smallsetminus \{ \epsilon_i \pm\epsilon_{j} \in I \mid 1\le i<r, i <j \le \ell\}$. 
Thus $R$ is an ideal of the root subsystem of $\Phi(D_\ell)$ of  type $D_{\ell-r+1}$ with a base given by $\Delta(D_{\ell-r+1})=\{\alpha_r,\ldots,\alpha_\ell\}$. 
Furthermore, for all $q \in \Z_{>0}$, we have
\begin{equation*}\label{eq:inductive-D}
\begin{aligned}
\chi^{\quasi}_{T_I}(D_\ell, q) & =\chi^{\quasi}_{T_R}(D_{\ell-r+1}, q)\cdot \prod_{i=1}^{r-1}(q-p_i),\\
F_I(D_\ell, q) &= F_R(D_{\ell-r+1}, q)\cdot \prod_{i=1}^{r-1}(q-p_i).
\end{aligned}
\end{equation*}
Then it suffices to consider $r=1$ i.e., $\epsilon_1\pm\epsilon_{\ell}  \in I$. 
For such ideals, $p_{i}^{(-)} = p_{i+1}^{(-)}+1=\ell-i$ for $1 \le i \le \ell-1$.
Moreover, the subset $\{ \epsilon_i \pm \epsilon_{j} \mid s-1 \le i<j \le \ell\}\subseteq I$ is the positive system of a root subsystem of $\Phi(D_\ell)$ of type $D_{\ell-s+2}$. 
Thus $p_i \le p_{i+1}+1$, $p_{i}^{(+)} \le p_{i+1}^{(+)}$ for all $1 \le i \le s-3$, and $p_i+2=p_{i-1}$ for $s \le i \le \ell$. 
We will need the following lemma.

 \begin{lemma}
 \label{lem:needed-B}
Let $\Psi$ be a subset of $\Phi^+(B_\ell)$ such that $\{ \epsilon_i \pm \epsilon_{j} \mid s-1 \le i<j \le \ell\}\subseteq \Psi$ for some $2 \le s \le \ell$. 
Assume that $\SG(\Psi)=(p_1,\ldots,p_\ell)$ with  $p_i  \le p_{i+1}+1$ for all $1 \le i \le s-3$.
Then $\Psi$ is an ideal of $\Phi^+(B_{\ell})$.
\end{lemma} 

\begin{proof} 
For $\beta_1,\beta_2 \in \Phi^+(B_\ell)$, $\beta_1 \succeq \beta_2$, $\beta_ 1 \in \Psi$, we will prove that $\beta_2 \in \Psi$. 
Note that for each $\beta\in \Phi^+(B_\ell)$, we have $\# \{ \gamma \in \Phi^+(B_\ell) \mid \beta - \gamma \in \Delta(B_\ell)\} \le 2$.
Since $\beta_1 \succeq \beta_2$, there exists a path in the Hasse diagram of $\Phi^+(B_\ell)$ connecting $\beta_1$ and $\beta_2$.\footnote{This fact is true for any root system, which is a consequence of, e.g., \cite[Lemma 3.2]{S05}.} 
It follows that this path must lie entirely within $\Psi$, yielding $\beta_2 \in \Psi$.
 \end{proof}

Let $\Pi$ be an irreducible root system of type $B_{\ell-1}$ with a base given by $\Delta=\{\alpha_i = \epsilon_i - \epsilon_{i+1}\,(1 \le i  \le \ell-2),\, \alpha_{\ell-1} =\epsilon_{\ell-1} \}$.
We define a sequence of subsets $\{U_k\}_{k=1}^\ell$ (depending on $I$) of $\Pi^+(B_{\ell-1})$ classified into two types as follows:
\begin{enumerate}[(i)]
\item Type I, 
   \begin{equation}\label{eq:1st-Uk}
   \SG(U_k)=(p_1,\ldots,p_{k-1}, \widehat{p_{k}}, p_{k+1}+1,\ldots,p_{\ell}+1), 
\end{equation}
for $1 \le k \le s-2$. Here $\widehat{p_{k}}$ means omission.

\item
Type II,
  \begin{equation}\label{eq:2nd-Uk}
  \SG(U_k)=(p_1-\tau_{1,k},\ldots,p_{s-2}-\tau_{s-2,k}, p_{s-1}-1,\ldots,p_{\ell-1}-1), 
\end{equation}
 for $s-1\le k \le \ell$,  $1 \le n \le s-2$, with
\begin{equation*}
\tau_{n,k}:=
\begin{cases}
0 \quad\mbox{ if $\epsilon_n+\epsilon_k  \notin I$} \\ 
1 \quad\mbox{ if $\epsilon_n+\epsilon_k   \in I$.}
\end{cases}
\end{equation*}
\end{enumerate}
It is easily seen that $\tau_{n,k} \le \tau_{n+1,k}$ (as well as $\tau_{n,k} \le \tau_{n,k+1}$), hence $p_{n}-\tau_{n,k} \le p_{n+1}-\tau_{n+1,k}+1$ for all $1 \le n \le s-3$.
By Lemma \ref{lem:needed-B}, the subsets $\{U_k\}_{k=1}^\ell$ are indeed ideals of $\Pi^+(B_{\ell-1})$. We also define
  \begin{equation}\label{eq:D-2nd-summand}
K:=I \sqcup \{\epsilon_k \mid 1 \le k \le \ell\}.
\end{equation}
Then again by Lemma \ref{lem:needed-B}, $K$ is an ideal of $\Phi^+(B_{\ell})$ with   
\begin{equation*}\label{eq:DP(K)}
   \SG(K)=(p_1+1, p_{2}+1,\ldots,p_{\ell}+1).
\end{equation*}
The following result is a generalization of \cite[Lemma 4.4(12)]{KTT07}.
\begin{lemma} \label{lem:D-to-B} 
Let $I$ be an ideal of $\Phi^+(D_\ell)$ so that $\epsilon_1\pm\epsilon_{\ell}  \in I$. 
For all $q \in \Z_{>0}$, we have
  \begin{equation}\label{eq:D-to-B}
\chi^{\quasi}_{T_{I}}(D_\ell, q)= \sum_{k=1}^{\ell}\chi^{\quasi}_{T_{U_k}}(B_{\ell-1}, q)+\chi^{\quasi}_{T_K}(B_{\ell}, q),
\end{equation}
where $U_k$ and $K$are defined in \eqref{eq:1st-Uk}, \eqref{eq:2nd-Uk}, \eqref{eq:D-2nd-summand}.
\end{lemma}
\begin{proof} 
With the notion of contraction lists (e.g., \cite[Section 2]{Tan18}), we can write $\chi^{\quasi}_{T_{U_k}}(B_{\ell-1}, q)=\chi^{\quasi}_{\A_k}(B_\ell, q)$ with $\A_k :=T_{I\cup\{\epsilon_\ell,\ldots,\epsilon_{k}\}}/T_{\{\epsilon_k\}}$ for $1 \le k \le \ell$. 
For all $q \in \Z_{>0}$, by applying the Deletion-Contraction formula \cite[Theorem 3.5]{Tan18} recursively, we get
\begin{align*}
\chi^{\quasi}_{T_{I}}(D_\ell, q)
& =  \chi^{\quasi}_{T_{U_\ell}}(B_{\ell-1}, q) + \chi^{\quasi}_{T_{I\cup\{\epsilon_\ell\}}}(B_\ell, q)\\
& =  \chi^{\quasi}_{T_{U_\ell}}(B_{\ell-1}, q) + \chi^{\quasi}_{T_{U_{\ell-1}}}(B_{\ell-1}, q) +\chi^{\quasi}_{T_{I\cup\{\epsilon_\ell, \epsilon_{\ell-1}\}}}(B_\ell, q)\\
& =\ldots\\
& =\sum_{k=1}^{\ell}\chi^{\quasi}_{T_{U_k}}(B_{\ell-1}, q)+\chi^{\quasi}_{T_K}(B_{\ell},q).
\end{align*}
\end{proof}

In Lemma \ref{lem:D-to-B-F} and Theorem \ref{thm:D-to-B-L} below, we use the same assumption and notation as in Lemma \ref{lem:D-to-B}.
\begin{lemma} \label{lem:D-to-B-F} 
For even $q \in \Z_{>0}$, we have
  \begin{equation*}\label{eq:D-to-B-F}
F_I(D_\ell, q) =F_{I\cup\{2\epsilon_s,\ldots,2\epsilon_{\ell}\}}(C_\ell, q)=\prod_{i=1}^{\ell}(q-p_i).
\end{equation*}
\end{lemma}
\begin{proof} 
This follows from the height placements in Table \ref{tab:placement}.
 \end{proof} 

\begin{theorem} \label{thm:D-to-B-L} 
For even $q \in \Z_{>0}$, we have
 \begin{equation*}\label{eq:D-to-B-L}
\chi^{\quasi}_{S_{I}}(D_\ell, q)= \frac12\left(\sum_{k=1}^{\ell}\chi^{\quasi}_{T_{U_k}}(B_{\ell-1}, q)+\chi^{\quasi}_{T_K}(B_{\ell}, q) +\prod_{i=1}^{\ell}(q-p_i) \right).
\end{equation*}
\end{theorem}
\begin{proof} 
This follows from formula \eqref{eq:D-TS}, and Lemmas \ref{lem:D-to-B},  \ref{lem:D-to-B-F}.
 \end{proof} 

\begin{example}\label{eg:D-ideal}   
 Table \ref{tab:D5-ideal-eg} shows an example of the ideal $I=\{\alpha \in \Phi^+(D_5) \mid {\rm ht}(\alpha) \le 6\}$ (in colored region), with positive roots contributing to $p_1$, $p_2$, $p_3$, $p_4$ are colored in green, yellow, blue, red, respectively. 
In this case, $s=2$ since $\epsilon_2+\epsilon_3 \in I$, but $\epsilon_1+\epsilon_2 \notin I$.
We have $\SG(I)=(7,6,4,2,0)$, and the computation on the ideals $K$ and $U_k$ for even $q \in \Z_{>0}$ is given in Table \ref{tab:data-D}. 
By Theorem \ref{thm:D-to-B-L}, for even $q \in \Z_{>0}$, we have
$$\chi^{\quasi}_{S_{I}}(D_\ell, q)=(q-2)(q-4)(q^3-13q^2+51q-51).$$
\begin{table}[htbp]
\centering
{\footnotesize\renewcommand\arraystretch{1.5} 
\begin{tabular}{K{1.5cm}|K{2.5cm}|K{2.5cm}}
Ideals & $\DP$ & Roots of $\chi^{\quasi}_{T}$  \\
\hline\hline
$K$ & $(8,7,5,3,1)$ & $7,6,5,4,2$ \\
\hline
$U_1, U_2$ & $(7,5,3,1)$ & $6,4,4,2$ \\
\hline
$U_3, U_4, U_5$ & $(6,5,3,1)$ & $5,4,4,2$ \\
\end{tabular}
}
\bigskip
\caption{Computation of Example \ref{eg:D-ideal}.}
\label{tab:data-D}
\end{table}
\end{example}
 
 \begin{table}[htbp]
\centering
{\footnotesize\renewcommand\arraystretch{1.5} 
\begin{tabular}{ccccccc}
\mbox{Height} & & & & &\\
7&   & $\epsilon_1+\epsilon_2$   &  & && \\
\hhline{~~-}
6&  &  \multicolumn{1}{|c|}{\cellcolor{green!50}{$\epsilon_1+\epsilon_3$}}   & & &&\\
\cline{4-4}
5 &&\multicolumn{1}{|c}{\cellcolor{green!50}{$\epsilon_1+\epsilon_4$}}   &\multicolumn{1}{c|}{\cellcolor{yellow!25}{$\epsilon_2+\epsilon_3$}}  && \\
\cline{2-2}
4& \multicolumn{1}{|c}{\cellcolor{green!50}$\epsilon_1-\epsilon_5$} &  \cellcolor{green!50}{$\epsilon_1+\epsilon_5$}  & \multicolumn{1}{c|}{\cellcolor{yellow!25}{$\epsilon_2+\epsilon_4$}} &&&\\
\cline{5-5}
3 & \multicolumn{1}{|c}{\cellcolor{green!50}$\epsilon_1-\epsilon_4$} & \cellcolor{yellow!25}$\epsilon_2-\epsilon_5$  &  \cellcolor{yellow!25}{$\epsilon_2+\epsilon_5$} &\multicolumn{1}{c|}{\cellcolor{blue!25}{$\epsilon_3 +\epsilon_4$}}  &&\\
2& \multicolumn{1}{|c}{\cellcolor{green!50}$\epsilon_1-\epsilon_3$} & \cellcolor{yellow!25}$\epsilon_2-\epsilon_4$ &  \cellcolor{blue!25}$\epsilon_3-\epsilon_5$  &\multicolumn{1}{c|}{\cellcolor{blue!25}{$\epsilon_3 +\epsilon_5$}}&& \\
\cline{6-6}
1 & \multicolumn{1}{|c}{\cellcolor{green!50}$\epsilon_1-\epsilon_2$} & \cellcolor{yellow!25}$\epsilon_2 -\epsilon_3$& \cellcolor{blue!25}$\epsilon_3-\epsilon_4$  &  \cellcolor{red!25}$\epsilon_4-\epsilon_5$ & \multicolumn{1}{c|}{\cellcolor{red!25}{$\epsilon_4+\epsilon_5$}} &\\
\cline{2-6}
& 7 & 6 & 4 & 2  & 0 & $\SG(I)$
\end{tabular}
}
\bigskip
\caption{$I=\{\alpha \in \Phi^+(D_5) \mid {\rm ht}(\alpha) \le 6\}$ in $\Phi^+(D_5)$.}
\label{tab:D5-ideal-eg}
\end{table}

\begin{remark}
\label{rem:D-recover}
When $I=\Phi^+(D_\ell)$, $\SG(I)=(2\ell-2,2\ell-4,\ldots,2,0)$, $\DP(I)=(\ell-1, 2\ell-3,\ldots,3,1)$, $\DP(K)=(2\ell-1,2\ell-3,\ldots,3,1)$, and 
$\DP(U_k)=(2\ell-3,\ldots,3,1)$ for all $1 \le k \le \ell$. 
Note that $s=2$, so there is no ideal $U_k$ of type I.
Then by Lemma \ref{lem:D-to-B}, for odd $q \in \Z_{>0}$
   \begin{equation*} 
\chi^{\quasi}_{S_{\Phi^+}}(D_\ell, q)=\chi^{\quasi}_{T_{\Phi^+}}(D_\ell, q)=(q-1)(q-3)\ldots(q-(2\ell-3))(q-(\ell-1)),
\end{equation*}
which agrees with Corollary \ref{cor:ideal-factorization}.
Moreover, for even $q \in \Z_{>0}$ 
   \begin{equation*} 
\chi^{\quasi}_{S_{\Phi^+}}(D_\ell, q)=(q-2)(q-4)\ldots(q-(2\ell-4))\left(q^2-2(\ell-1)q+\frac{\ell(\ell-1)}2\right),
\end{equation*}
which recovers the result of \cite[Theorem 4.8]{KTT07} for type $D$ root systems.
 \end{remark}

With a recent study on characteristic quasi-polynomials and toric arrangements \cite[Corollary 5.6]{LTY17}, our computation gives a full description of the characteristic polynomials of the toric arrangements defined by the ideals.
We complete this section by giving a direct verification of Corollary \ref{cor:ideal-factorization} when $\Phi$ is any classical root system. 
We restrict the discussion to type $D$ root systems as the other cases are easy.
For any ideal $I\subseteq\Phi^+(D_\ell)$ with $\SG(I)=(p_1,\ldots,p_\ell)$ defined in \eqref{eq:di-signed-graphs-D}, we write 
$$p_i=p_i^{(+)}+p_i^{(-)}, \mbox{where, } p_i^{(\pm)}:=\#\{ \epsilon_i \pm\epsilon_{j} \mid \epsilon_i \pm\epsilon_{j} \in I \},$$
for each $1 \le i \le \ell$. 
It is easily seen that $\DP(I)=(d_1,\ldots,d_\ell)$ with 
$$d_i=p_i^{(-)}+p_{i-1}^{(+)}.$$
Here we agree that $p_{0}^{(+)} = 0$.

\begin{theorem}
\label{thm:verify}
Let $I$ be an ideal of $\Phi^+(D_\ell)$. For odd $q \in \Z_{>0}$, we have
  \begin{equation}
  \label{eq:verify}
\chi^{\quasi}_{T_{I}}(D_\ell, q)= \prod_{i=1}^{\ell}\left(q- d_i  \right).
   \end{equation}
\end{theorem}
 \begin{proof} 
It suffices to prove Theorem \ref{thm:verify} when $\epsilon_1\pm\epsilon_{\ell}  \in I$, as the other cases are straightforward.  
For such ideals, $d_1=\ell-1$, $d_i=p_i^{(-)}+p_{i-1}^{(+)}=p_{i-1}-1$ for all  $2 \le i \le \ell$. 
We recall the notation of the parameter $s$ defined in \eqref{eq:s-D} that $s =\min \{ 2 \le k \le \ell \mid \epsilon_{k-1}+\epsilon_k  \in I\}$. 
It follows from Lemma \ref{lem:D-to-B} and Remark \ref{rem:D-recover} that both sides of \eqref{eq:verify} are divisible by $\prod_{i=s}^{\ell}\left(q- p_{i-1}+1 \right)$. 
Hence we need only prove the following:
  \begin{equation}
  \label{eq:verify-need}
  A+B+C=(q-\ell+1)\prod_{i=2}^{s-1}\left(q- p_{i-1}+1\right),
     \end{equation}
where
\begin{align*}
A & := \prod_{i=1}^{s-1}\left(q- p_i-1 \right), \\
B & :=  \sum_{k=1}^{s-2}(q-p_1)\ldots(q-p_{k-1})(q-p_{k+1}-1)\ldots(q-p_{s-1}-1), \\
C & = \sum_{k=s-1}^{\ell}C_k, \mbox{ with } C_k:= \prod_{n=1}^{s-2}\left(q- p_n+\tau_{n,k} \right),
\end{align*}
and $\tau_{n,k}$ is defined in  \eqref{eq:2nd-Uk}. 
Since $\tau_{n,s-1}=0$ for all $1 \le n \le s-2$, $C_{s-1}=\prod_{i=1}^{s-2}\left(q- p_i\right)$. 
It is routine to check that
  \begin{equation}
  \label{eq:verify-routine}
  A+B+C_{s-1}=\prod_{i=1}^{s-1}\left(q- p_i \right).
       \end{equation}
Write $M_\tau=[\tau_{n,k}]$ for a matrix of size $(s-2) \times (\ell-s+1)$ whose entries are the $\tau_{n,k}$'s (the columns indexed by the set $\{s,\ldots,\ell\}$).
Then
$$
M_\tau =
\begin{bmatrix}
0 & \cdots & 0  & \cdots  &\cdots  & 0  & 1 & \cdots  & 1  \\
0 & \cdots & 0  & \cdots  &1 & \cdots   & 1 & \cdots  & 1  \\
\vdots  &  & \vdots  & & \vdots  &    & \vdots  &   & \vdots   \\
0 & \cdots & 0 & 1 & \cdots & \cdots  & 1 & \cdots  & 1  \\
\end{bmatrix},
$$
with the number of $1$'s on the $n$-th row is exactly $p_{n}^{(+)}$, and the entries on the $k$-th column contribute to the evaluation of $C_k$.
Thus
  \begin{equation}
  \label{eq:verify-Ck}
 \sum_{k=s}^{\ell}C_k = \sum_{n=0}^{s-2}\left(p_{n+1}^{(+)}-p_{n}^{(+)} \right)\prod_{i=1}^{n}\left(q- p_i \right)\prod_{i=n+1}^{s-2}\left(q- p_i+1 \right).
       \end{equation}
Now combining \eqref{eq:verify-routine} and \eqref{eq:verify-Ck} with a rigorous check, we obtain \eqref{eq:verify-need}.

 \end{proof}

\noindent
\textbf{Acknowledgements:} 
The author is greatly indebted to Professor Masahiko Yoshinaga for drawing the author's attention to the characteristic quasi-polynomials of the ideals and for many helpful suggestions during the preparation of the paper. 
The author wishes to thank Professor Michele Torielli for helpful comments concerning the signed graphs and thank Ye Liu for stimulating conversations. 
He also gratefully acknowledges the support of the scholarship program of 
the Japanese Ministry of Education, Culture, Sports, Science, and Technology 
(MEXT) under grant number 142506. 

\bibliographystyle{alpha} 
\bibliography{references}

\newcommand{\etalchar}[1]{$^{#1}$}
\begin{thebibliography}{ABC{\etalchar{+}}16}

\bibitem[ABC{\etalchar{+}}16]{ABCHT16}
T.~Abe, M.~Barakat, M.~Cuntz, T.~Hoge, and H.~Terao.
\newblock The freeness of ideal subarrangements of {W}eyl arrangements.
\newblock {\em J. Eur. Math. Soc.}, 18:1339--1348, 2016.

\bibitem[ACH15]{ACH15}
F.~Ardila, F.~Castillo, and M.~Henley.
\newblock The arithmetic {T}utte polynomials of the classical root systems.
\newblock {\em International Mathematics Research Notices},
  2015(12):3830--3877, 2015.

\bibitem[Ath96]{A96}
C.~A. Athanasiadis.
\newblock Characteristic polynomials of subspace arrangements and finite
  fields.
\newblock {\em Adv. Math.}, 122:193--233, 1996.

\bibitem[Bou68]{B68}
N.~Bourbaki.
\newblock {\em Groupes et {A}lg\`ebres de {L}ie}.
\newblock Chapitres 4,5 et 6, Hermann, Paris, 1968.

\bibitem[BS98]{BS98}
A.~Blass and B.~Sagan.
\newblock Characteristic and {E}hrhart polynomials.
\newblock {\em J. Algebr. Comb.}, 7:115--126, 1998.

\bibitem[DFM17]{DFM17}
C.~Dupont, A.~Fink, and L.~Moci.
\newblock Universal {T}utte characters via combinatorial coalgebras.
\newblock {\em arXiv preprint}, 2017.
\newblock \url{https://arxiv.org/abs/1711.09028v1}.

\bibitem[KTT07]{KTT07}
H.~Kamiya, A.~Takemura, and H.~Terao.
\newblock The characteristic quasi-polynomials of the arrangements of root
  systems and mid-hyperplane arrangements.
\newblock {\em arXiv preprint [v1]}, 2007.
\newblock \url{http://arxiv.org/abs/0707.1381}.

\bibitem[KTT08]{KTT08}
H.~Kamiya, A.~Takemura, and H.~Terao.
\newblock Periodicity of hyperplane arrangements with integral coefficients
  modulo positive integers.
\newblock {\em J. Alg. Combin.}, 2008.

\bibitem[KTT10]{KTT10}
H.~Kamiya, A.~Takemura, and H.~Terao.
\newblock The characteristic quasi-polynomials of the arrangements of root
  systems and mid-hyperplane arrangements.
\newblock {\em Arrangements, local systems and singularities}, pages 177--190,
  2010.

\bibitem[LTY17]{LTY17}
Y.~Liu, T.~N. Tran, and M.~Yoshinaga.
\newblock {$G$}-{T}utte polynomials and abelian {L}ie group arrangements.
\newblock {\em arXiv preprint}, 2017.
\newblock \url{https://arxiv.org/abs/1707.04551v2}.

\bibitem[OT92]{OT92}
P.~Orlik and H.~Terao.
\newblock {\em Arrangements of hyperplanes}.
\newblock Grundlehren der Mathematischen Wissenschaften 300, Springer-Verlag,
  Berlin, 1992.

\bibitem[Som05]{S05}
E.~N. Sommers.
\newblock {$B$}-stable ideals in the nilradical of a {B}orel subalgebra.
\newblock {\em Canadian Mathematical Bulletin}, 48(3), 2005.

\bibitem[Sut98]{Su98}
R.~Suter.
\newblock The number of lattice points in alcoves and the exponents of the
  finite {W}eyl groups.
\newblock {\em Math. Comp.}, 67(222):751--758, 1998.

\bibitem[Tra18]{Tan18}
T.~N. Tran.
\newblock An equivalent formulation of chromatic quasi-polynomials.
\newblock {\em arXiv preprint}, 2018.
\newblock \url{https://arxiv.org/abs/1803.08649}.

\bibitem[TY18]{TY18}
T.~N. Tran and M.~Yoshinaga.
\newblock Combinatorics of abelian {L}ie group arrangements and chromatic
  quasi-polynomials.
\newblock {\em In preparation}, 2018.

\bibitem[Yos18a]{Y18L}
M.~Yoshinaga.
\newblock Characteristic polynomials of {L}inial arrangements for exceptional
  root systems.
\newblock {\em J. of Combinatorial Theory, Series A}, 157:267--286, 2018.

\bibitem[Yos18b]{Y18W}
M.~Yoshinaga.
\newblock {W}orpitzky partitions for root systems and characteristic
  quasi-polynomials.
\newblock {\em Tohoku Math. J.}, 70(1):39--63, 2018.

\bibitem[Zas81]{Z81}
T.~Zaslavsky.
\newblock The geometry of root systems and signed graphs.
\newblock {\em The American Mathematical Monthly}, 88(2):88--105, 1981.

\end{thebibliography}

\end{document}